\documentclass[11pt,reqno]{amsart}
\usepackage{amssymb,latexsym,graphicx,amscd,upgreek}
\usepackage{lscape}
\usepackage[all]{xy}
\usepackage{color}
\usepackage{epic,eepic}
\usepackage{xspace}
\usepackage{mathrsfs}
\usepackage{setspace}
\usepackage{cases}
\usepackage{bbm}
\newtheorem{Thm}{Theorem}[section]
\newtheorem{Lem}[Thm]{Lemma}
\newtheorem{Cor}[Thm]{Corollary}
\newtheorem{Prop}[Thm]{Proposition}
\newtheorem{Conj}[Thm]{Conjecture}

\newtheorem{Problem}[Thm]{Problem}

\theoremstyle{definition}

\newcommand{\N}{\mathbb{N}}

\newcommand{\df}{\colon}
\newcommand{\id}{\operatorname{id}}

\newcommand{\cI}{{\mathcal I}}

\newcommand{\cO}{{\mathcal O}}

\newcommand{\cS}{{\mathcal S}}
\newcommand{\cT}{{\mathcal T}}
\newcommand{\cU}{{\mathcal U}}

\newcommand{\cX}{{\mathcal X}}
\newcommand{\cY}{{\mathcal Y}}

\newcommand{\bd}{\mathbf{d}}

\newcommand{\bp}{{\mathbf p}}

\newcommand{\vep}{\varepsilon}

\newcommand{\rk}{\operatorname{rank}}

\newcommand{\md}{\operatorname{mod}}

\newcommand{\rep}{\operatorname{rep}}

\newcommand{\gldim}{\operatorname{gl.dim}}
\newcommand{\pdim}{\operatorname{proj.dim}}
\newcommand{\idim}{\operatorname{inj.dim}}

\newcommand{\dimv}{\underline{\dim}}

\newcommand{\soc}{\operatorname{soc}}
\newcommand{\rad}{\operatorname{rad}}
\newcommand{\tp}{\operatorname{top}}

\newcommand{\Hom}{\operatorname{Hom}}
\newcommand{\Ext}{\operatorname{Ext}}
\newcommand{\End}{\operatorname{End}}

\newcommand{\supp}{\operatorname{supp}}
\newcommand{\length}{\operatorname{length}}
\newcommand{\ov}{\overline}

\newcommand{\bsm}{\begin{smallmatrix}}
\newcommand{\esm}{\end{smallmatrix}}

\newcommand{\bbsm}{\left[\begin{smallmatrix}}
\newcommand{\besm}{\end{smallmatrix}\right]}

\newcommand{\bbm}{\begin{matrix}}
\newcommand{\ebm}{\end{matrix}}

\newcommand{\GL}{\operatorname{GL}}

\begin{document}

\date{18.09.2017}

\title[Algebras with irreducible module varieties I]{Algebras with irreducible module varieties I}

\author{Grzegorz Bobi\'nski}
\address{Grzegorz Bobi\'nski\newline
Faculty of Mathematics and Computer Science\newline
Nicolaus Copernicus University\newline
ul. Chopina 12/18\newline
87-100 Toru\'n\newline
Poland}
\email{gregbob@mat.umk.pl}

\author{Jan Schr\"oer}
\address{Jan Schr\"oer\newline
Mathematisches Institut\newline
Universit\"at Bonn\newline
Endenicher Allee 60\newline
53115 Bonn\newline
Germany}
\email{schroer@math.uni-bonn.de}



\begin{abstract}
We call a finite-dimensional $K$-algebra $A$ \emph{geometrically irreducible}
if for all $d \ge 0$ all
connected components of the affine scheme of $d$-dimensional $A$-modules are irreducible.
We show that the geometrically irreducible algebras without loops
(this includes all algebras of finite global dimension) are
the hereditary algebras.
We also show that truncated polynomial rings are the only
geometrically irreducible local algebras.
Finally, we formulate a conjectural classification of all geometrically
irreducible algebras.
\end{abstract}

\maketitle

\setcounter{tocdepth}{1}
\numberwithin{equation}{section}
\tableofcontents

\parskip2mm


\section{Introduction and main result}


\subsection{Introduction}
Let $A$ be a finite-dimensional $K$-algebra with $K$ an algebraically
closed field, and for $d \ge 0$ let
$\md(A,d)$ be the affine scheme of $d$-dimensional $A$-modules.

The connected components of $\md(A,d)$ correspond to the different
dimension vectors (which count the Jordan--H\"older multiplicities)
of $d$-dimensional modules.

In general, the connected components of $\md(A,d)$ are reducible and
singular.
The following result due to Bongartz \cite{Bo} determines when $\md(A,d)$
is smooth (i.e. non-singular) for all $d$, see also \cite{AGV} for
related aspects.
Recall that $A$ is \emph{hereditary} if $\gldim(A) \le 1$.

\begin{Thm}[Bongartz \cite{Bo}]\label{thm:intro1}
The following are equivalent:
\begin{itemize}

\item[(i)]
$\md(A,d)$ is smooth for all $d$;

\item[(ii)]
$A$ is hereditary.

\end{itemize}
\end{Thm}

We call $A$ \emph{geometrically irreducible} if for each $d$ all connected components of $\md(A,d)$ are irreducible.
Note that Theorem~\ref{thm:intro1} implies that all hereditary algebras $A$ are geometrically irreducible.

In this article, we deal with the following problem.

\begin{Problem}\label{problem:intro}
Find a necessary and sufficient condition for $A$ to be geometrically
irreducible.
\end{Problem}

We did not find a complete answer to Problem~\ref{problem:intro}.
But we have a list of partial results and also a general conjecture.

Thanks to the work of Bongartz \cite{Bo} on the behaviour of $\md(A,d)$
under Morita equivalence, we can assume without loss of generality that
$A = KQ/\cI$, where $KQ$ is the path algebra of a quiver $Q$ and
$\cI$ is an admissible ideal in $KQ$.
For a dimension vector $\bd \in \N^n$, where $n$ is the
number of vertices of $Q$, let $\rep(A,\bd)$ be the affine scheme of
representations of $Q$ which are annihilated by $\cI$.
Then $A$ is geometrically irreducible if and only if $\rep(A,\bd)$ is
irreducible for all $\bd$.

\subsection{Main result}\label{sec:intro1.2}
Recall that an arrow $\vep$ in $Q$ is a \emph{loop} if
$\vep$ starts and ends in the same vertex.

\begin{Thm}\label{thm:intro2a}
Let $A = KQ/\cI$.
Assume that $Q$ does not contain any loop.
Then the following are equivalent:
\begin{itemize}

\item[(i)]
$A$ is geometrically irreducible;

\item[(ii)]
$A$ is hereditary.

\end{itemize}
\end{Thm}

As a relatively straightforward consequence of Theorem~\ref{thm:intro2a}
combined with the No-Loop Theorem (cf. \cite{I} and \cite{L})
we get the following characterization of geometrically irreducible algebras
of finite global dimension.

\begin{Cor}\label{cor:intro2b}
Assume that $\gldim(A) < \infty$.
Then the following are equivalent:
\begin{itemize}

\item[(i)]
$A$ is geometrically irreducible;

\item[(ii)]
$A$ is hereditary.

\end{itemize}
\end{Cor}

An easy example of geometrically irreducible and non-hereditary algebras are truncated polynomial rings $K[X]/(X^m)$ with $m \ge 2$.
(Note that these are isomorphic to the algebra $KQ/\cI$, where $Q$
has exactly one vertex and one loop $\vep$, and $\cI$ is generated by
$\vep^m$.)
The following proposition shows that these are the only examples of
geometrically irreducible local algebras.

\begin{Prop}\label{prop:intro2c}
Assume that $A = KQ/\cI$ is local, i.e. $Q$ has just one vertex.
Then the following are equivalent:
\begin{itemize}

\item[(i)]
$A$ is geometrically irreducible;

\item[(ii)]
$Q$ has at most one arrow;

\item[(iii)]
$A \cong K[X]/(X^m)$
for some $m \ge 1$.

\end{itemize}
\end{Prop}

Note that each finite-dimensional local $K$-algebra
is of the form $KQ/\cI$.

Back to the general case, assume that
$A = KQ/\cI$ is geometrically irreducible.
It is not hard to show that each oriented cycle in $Q$ is a power
of a loop.
In particular, there is at most one loop at each vertex.

Let $Q$ be the quiver
\[
\xymatrix{
0 \ar@(ul,dl)_{\vep_0}& 1 \ar@(ur,dr)^{\vep_1} \ar[l]_{\alpha}
}.
\]
For $n \ge 2$ let
$$
\rho_n := \sum_{p=0}^{n-1} \vep_0^{n-p-1}\alpha\vep_1^p.
$$
For $m,n \ge 2$
let $B(m,n) := KQ/\cI$, where the ideal $\cI$ is
generated by $\{ \vep_0^m, \vep_1^m, \rho_n \}$.

Up to some trivial glueing procedure explained in Section~\ref{sec:glue},
we conjecture that hereditary algebras, truncated polynomial rings, and the algebras $B(m,2)$ and $B(m,m)$ are the only geometrically irreducible algebras.
In \cite{BS1} we prove that (modulo glueing) geometrically irreducible algebras $A = KQ/\cI$, where $Q$ has at most two vertices,
belong to one of the four classes mentioned above, and in \cite{BS2} we show that the algebras $B(m,2)$ are indeed
geometrically irreducible.

\subsection{Finitely generated algebras}
Instead of studying finite-dimensional algebras $A$, we could
assume that $A$ is finitely generated, and ask when all connected
components of $\md(A,d)$ are irreducible.
The following example shows that this more general classification problem
includes some interesting classical examples.

Namely,
let $A = K[X,Y]$ be the polynomial ring in two commuting variables.
Then
$$
\md(A,d) \cong \{ (X,Y) \in M_d(K) \times M_d(K) \mid AB = BA \}
$$
is the famous \emph{commuting variety}.
It follows from \cite[Theorem~6]{MT} that
$\md(A,d)$ is irreducible for all $d$.
So $A$ is geometrically irreducible.
On the other hand, $A = K[X,Y,Z]$ is geometrically reducible,
see for example \cite[Theorem~3(b)]{Gu}.

One should also keep in mind that there are numerous finitely generated
algebras $A$ which do not have any finite-dimensional modules.
In this case, $\md(A,d)$ is empty for all $d$.

\subsection{Notation}
Throughout, let $K$ be an algebraically closed field.
By a \emph{module} we mean a finite-dimensional
left module, if not mentioned otherwise.


\section{Schemes of representations}


\subsection{Quivers with relations}
A \emph{quiver} is a quadruple $Q = (Q_0,Q_1,s,t)$ where
$Q_0$ and $Q_1$ are finite sets of \emph{vertices} and
\emph{arrows}, respectively, and $s,t\df Q_1 \to Q_0$ are
maps.
We say that an arrow $a$ \emph{starts} in $s(a)$ and ends in $t(a)$.

A \emph{path} of length $m \ge 1$ in $Q$ is a sequence
$p = (a_1,\ldots,a_m)$ of arrows $a_1,\ldots,a_m \in Q_1$ such that
$s(a_i) = t(a_{i+1})$ for all $1 \le i \le m-1$.
We set $\length(p) := m$.
Define $s(p) := s(a_m)$ and $t(p) := t(a_1)$.
Let
$$
\supp(p) := \{ s(a_1),\ldots,s(a_m),t(a_1) \}
$$
be the \emph{support} of $p$, and let $Q_p$ be the subquiver of $Q$ with vertices $\supp(p)$ and arrows $\{ a_1,\ldots,a_m \}$.
An arrow $a$ \emph{occurs} in $p$ if $a = a_i$ for some $1 \le i \le m$.
The path $p$ is a
\emph{cycle} if we have additionally $s(p) = t(p)$.
We usually write $a_1 \cdots a_m$ instead of $(a_1,\ldots,a_m)$.
A
\emph{loop} is a cycle of length one.
Clearly, the $t$-fold compositon $p^t := p \cdots p$ of a cycle $p$ is again a cycle.
A cycle $(a_1,\ldots,a_m)$ is \emph{primitive}
if the vertices $s(a_1),\ldots,s(a_m)$ are pairwise different.

\begin{Lem}\label{lem:primitivecycles}
For each non-primitive cycle $(a_1,\ldots,a_m)$ there exists some
$1 \le u \le v \le m$
such that $v - u < m - 1$, $(a_u,\ldots,a_v)$ is a primitive cycle and
$(a_1,\ldots,a_{u-1},a_{v+1},\ldots,a_m)$ is a cycle.
\end{Lem}

\begin{proof}
Let $(a_1,\ldots,a_m)$ be a cycle in $Q$ which is not primitive.
Thus there exists some
$1 \leq u \leq v \leq m$ such that $s(a_v) = t(a_u)$ and
$$
v - u = \min \{ j - i \mid \text{$1 \leq i \leq j \leq m$ and
$s(a_j) = t(a_i)$} \} < m-1.
$$
Then $(a_u,\ldots,a_v)$ is a primitive cycle,
and the sequence $(a_1, \ldots, a_{u-1},a_{v+1},\ldots,a_m)$
is a cycle.
\end{proof}

Additionally to the paths of length $m \ge 1$, there is a path $e_i$ of length $0$ with $s(e_i) = t(e_i) = i$ for each vertex $i \in Q_0$.
Let $KQ$ be the path algebra of $Q$.
(The set of all paths (including the length $0$ paths) form a $K$-basis
of $KQ$, and the multiplication is defined via the composition of paths.)

A \emph{relation} for $Q$ is a linear combination
$$
\rho = \sum_{i=1}^t \lambda_i p_i
$$
where $\lambda_i \in K^*$ and the $p_i$ are pairwise different
paths of length at least two in $Q$ such that $s(p_i) = s(p_j)$ and
$t(p_i) = t(p_j)$ for all $1 \le i,j \le t$.
We say that a path $p$ \emph{occurs} in $\rho$ if $p = p_i$
for some $i$, and an arrow $a$ \emph{occurs} in $\rho$
if $a$ occurs in some $p_i$.
Let $Q_\rho$ be the union of the subquivers $Q_{p_1},\ldots,Q_{p_t}$.

An ideal $\cI$ in $KQ$ is \emph{admissible} if the following hold:
\begin{itemize}

\item[(i)]
$\cI$ is generated by a set $R = \{ \rho_1,\ldots,\rho_m \}$ of relations;

\item[(ii)]
There exists some $m \ge 2$ such that all paths of length at least $m$
are contained in $\cI$.

\end{itemize}
From now we assume that $A = KQ/\cI$ with $Q$ a quiver and $\cI$
an admissible ideal.

\subsection{Schemes of representations}
Let $A = KQ/\cI$ as before.
Assume that $Q_0 = \{ 1,\ldots,n \}$.
A \emph{representation} of $Q$ is a tuple
$M = (M(i),M(a))$, where $M(i)$ is a
finite-dimensional $K$-vector space for each vertex $1 \le i \le n$, and
$M(a)\df M(s(a)) \to M(t(a))$ is a $K$-linear map for each arrow
$a \in Q_1$.
For a path $p = (a_1,\ldots,a_m)$ in $Q$ and a representation
$M = (M(i),M(a))$ define
$$
M(p) := M(a_1) \cdots M(a_m),
$$
and for a relation $\rho = \sum_{i=1}^t \lambda_i p_i$ set
$$
M(\rho) = \sum_{i=1}^t \lambda_i M(p_i).
$$
A representation $M = (M(i),M(a))$ of $Q$ is a \emph{representation} of $A$
if $M(\rho) = 0$ for all relations $\rho \in \cI$.

For $\bd = (d_1,\ldots,d_n)$ let
$$
\rep(Q,\bd) :=  \prod_{a \in Q_1} \Hom_K(K^{d_{s(a)}},K^{d_{t(a)}})
$$
be the affine space of representations $M = (M(i),M(a))$ of $Q$ with
$M(i) = K^{d_i}$ for $1 \le i \le n$, and let
$$
\rep(A,\bd) = \{ M = (M(a)) \in \rep(Q,\bd) \mid
M(\rho) = 0 \text{ for all relations $\rho$ in } \cI
\}.
$$
Clearly, $\rep(A,\bd)$ is an affine scheme.

The group
$$
G(\bd) := \prod_{i \in Q_0} \GL_{d_i}(K)
$$
acts on $\rep(A,\bd)$ by conjugation.
The $G(\bd)$-orbits are in bijection with the set of isomorphism
classes of $A$-modules with dimension vector $\bd$.

Recall that an $A$-module $M$ is \emph{rigid} if $\Ext_A^1(M,M) = 0$.
It follows from Voigt's Lemma (see for example \cite[Proposition~1.1]{Ga})
that for a rigid module $M$ in $\rep(A,\bd)$, the orbit $\cO_M$ of $M$ is
open.
In particular, in this case the closure of $\cO_M$ is an irreducible component of $\rep(A,\bd)$.
Consequently, if $A$ is geometrically irreducible and $M$ and $N$ are rigid $A$-modules with the same dimension vector, 
then $M \cong N$. 

For some basic properties of the schemes $\md(A,d)$ and $\rep(A,\bd)$ we refer to \cite{Bo} and \cite{Ga}.

\subsection{Minimal algebras}\label{sec:glue}
We call $A = KQ/\cI$ \emph{minimal} if $\cI = 0$, or if for each set
$R = \{ \rho_1,\ldots,\rho_m \}$ of relations for $Q$ such that $R$ generates $\cI$ the following hold:
\begin{itemize}

\item[(i)]
For each $1 \le k \le m-1$ the quiver
$$
(Q_{\rho_1} \cup \cdots \cup Q_{\rho_k}) \cap
(Q_{\rho_{k+1}} \cup \cdots \cup Q_{\rho_m})
$$
contains at least one arrow.

\item[(ii)]
$Q = Q_{\rho_1} \cup \cdots \cup Q_{\rho_m}$.

\end{itemize}

If $A = KQ/\cI$ is not minimal, then there are algebras
$A' = KQ'/\cI'$ and $A'' = KQ''/\cI''$
with $Q'$ and $Q''$ proper subquivers of $Q$ such that
the following hold:
\begin{itemize}

\item[(a)]
$Q_0 = Q_0' \cup Q_0''$ (this union is not necessarily disjoint);

\item[(b)]
$Q_1 = Q_1' \cup Q_1''$ and $Q_1' \cap Q_1'' = \varnothing$;

\item[(c)]
For each dimension vector $\bd = (d_1,\ldots,d_n)$ we have
$$
\rep(A,\bd) \cong \rep(A',\bd') \times \rep(A'',\bd'').
$$
Here $\bd'$ and $\bd''$ are obtained from $\bd$ by deleting the
entries $d_i$ with $i$ not a vertex of $Q'$, respectively $Q''$.

\end{itemize}

More precisely, if there is an arrow
$a \in Q_1$ with $a \notin Q_{\rho_1} \cup \cdots \cup Q_{\rho_m}$, then $A' := KQ'/\cI'$, 
where $Q'$ is obtained from $Q$ by deleting $a$ and $\cI' := \cI \cap KQ'$, 
and $A'' := KQ''$, where $Q''$ is the subquiver of $Q$ with vertices $\{ s(a),t(a) \}$
and a single arrow $a$.
Next, assume (ii) holds, and assume that $(Q_{\rho_1} \cup \cdots \cup Q_{\rho_k}) \cap (Q_{\rho_{k+1}} \cup \cdots \cup Q_{\rho_m})
$
contains no arrow.
Then $Q' := Q_{\rho_1} \cup \cdots \cup Q_{\rho_k}$, $\cI' := (\rho_1,\ldots,\rho_k)$, $Q'' := Q_{\rho_{k+1}} \cup \cdots \cup Q_{\rho_m}$
and $\cI'' := (\rho_{k+1},\ldots,\rho_m)$.

Thus, if we want to classify geometrically irreducible algebras, we can
restrict our attention to the minimal ones.

\subsection{Classification of minimal geometrically irreducible algebras}

\begin{Conj}\label{conj:main}
Assume that $A = KQ/\cI$ is minimal geometrically irreducible and
that $\cI \not= 0$.
Then $Q$ has at most two vertices.
\end{Conj}

In \cite{BS1} we prove Conjecture~\ref{conj:main} for a large
class of quivers $Q$.

Recall Proposition~\ref{prop:intro2c} describes all geometrically irreducible algebras $A = KQ/\cI$ with $Q$ having just one vertex.

The following theorem deals with the two vertex case and is the main result of \cite{BS1}.

\begin{Thm}\label{thm:twovertices}
Assume that $A = KQ/\cI$ is minimal geometrically irreducible
and that $\cI \not= 0$,
and assume that $Q$ has at most two vertices.
Then (up to isomorphism), $A$ is defined by one of the
following quivers with relations:
\begin{itemize}

\item[(i)]
$Q = \xymatrix{\bullet \ar@(ul,dl)_{\vep}}$ and
$\cI = (\vep^m)$ for some $m \ge 2$;\vspace{0.3cm}

\item[(ii)]
$Q = \xymatrix{\bullet \ar@(ul,dl)_{\vep_0} & \bullet \ar[l]_\alpha\ar@(ur,dr)^{\vep_1}}$
and
$\cI = (\vep_0^m,\vep_1^m,\rho_2)$
for some $m \ge 2$;\vspace{0.3cm}

\item[(iii)]
$Q = \xymatrix{\bullet \ar@(ul,dl)_{\vep_0} & \bullet \ar[l]_\alpha\ar@(ur,dr)^{\vep_1}}$
and
$\cI = (\vep_0^m,\vep_1^m,\rho_m)$
for some $m \ge 3$.

\end{itemize}
\end{Thm}

The relations $\rho_m$ appearing in parts (ii) and (iii) of
Theorem~\ref{thm:twovertices} are defined in Section~\ref{sec:intro1.2}.

Clearly, the algebras mentioned in (i) are geometrically irreducible.
We show in \cite{BS2} that the algebras listed in (ii) are geometrically irreducible.
It remains open if the algebras in (iii) are geometrically irreducible
as well.


\section{Cycles and loops}


As before, let $A = KQ/\cI$ be a finite-dimensional $K$-algebra
with $Q = (Q_0,Q_1,s,t)$ a finite quiver and $\cI$ an admissible ideal
in the path algebra $KQ$.
Let $Q_0 = \{ 1,\ldots,n \}$.

Clearly, if $A$ is geometrically irreducible, then for all subsets
$E \subseteq Q_0$ we know that
$$
A_E := A/Ae_EA
$$
is also geometrically irreducible, where
$$
e_E := \sum_{i \in E} e_i.
$$

\begin{Lem}\label{lem:reduction1}
Assume that $A$ is geometrically irreducible.
Then any primitive cycle in $Q$ is a loop.
\end{Lem}

\begin{proof}
Assume that $p = (a_1,\ldots,a_m)$ is a primitive cycle in $Q$
which is not a loop.
Thus we have $\supp(p) = \{ s(a_1),\ldots,s(a_m) \}$ and $m \ge 2$.
Let $\bd = (d_1,\ldots,d_n)$ with
$$
d_i =
\begin{cases}
1 & \text{if $i \in \supp(p)$},\\
0 & \text{otherwise}.
\end{cases}
$$
We claim that $\rep(A,\bd)$ is reducible.
Clearly, for each $1 \le i \le m$ there exists a representation
$M_i \in \rep(A,\bd)$ such that for $a \in Q_1$ we have
$$
M_i(a) =
\begin{cases}
\id_K & \text{if $a = a_i$},\\
0 & \text{otherwise}.
\end{cases}
$$
Thus the open set
$\cU_i := \{ M \in \rep(A,\bd) \mid \rk(M(a_i)) \ge 1 \}$
is non-empty.
However, the intersection $\bigcap_{i=1}^m \cU_i$ has to be empty.
Otherwise, we would get some $M \in \rep(A,\bd)$ with
$M((a_1 \cdots a_m)^t)\not= 0$ for all $t \ge 1$, a contradiction
since $A$ is finite-dimensional.
Thus $\rep(A,\bd)$ cannot be irreducible.
\end{proof}

\begin{Lem}\label{lem:reduction2}
Assume that $A$ is geometrically irreducible.
Then for each vertex $i$ of $Q$ there is at most one
loop $\vep$ with $s(\vep) = i$.
\end{Lem}

\begin{proof}
We can assume without loss of generality that $Q$ has just one vertex,
and that $Q$ has $n \ge 2$ loops.
Let $P := {_A}A$ and $I := D(A_A)$ (here $D := \Hom_K(-,K)$ denotes the usual $K$-dual) and let $S$ be the simple $A$-module.
Set $d := \dim(P) = \dim(I) = \dim(A)$.
Since $A$ is geometrically irreducible and $P$ and $I$ are both rigid, we get $P \cong I$.
Thus $A$ is selfinjective.
We consider $\rad(P)$ and $P/\soc(P)$.
Since the ideal $\cI$ is by definition generated by paths of length at least
two, we know that $\tp(\rad(P)) \cong S^n$.
Set $\bd = (d-1)$.
We have
$$
\cS :=
\{ M \in \rep(A,\bd) \mid \soc(M) \cong S \} =
\{ M \in \rep(A,\bd) \mid M \cong \rad(P) \}
$$
and
$$
\cT :=
\{ M \in \rep(A,\bd) \mid \tp(M) \cong S \} =
\{ M \in \rep(A,\bd) \mid M \cong P/\soc(P) \} .
$$
The sets $\cS$ and $\cT$ are both non-empty and
open in $\rep(A,\bd)$.
Thus they have to coincide.
It follows that $\tp(\rad(P))$ is simple, a contradiction.
\end{proof}

\begin{Cor}\label{cor:reduction2}
Assume that $A = KQ/\cI$ is local.
Then the following are equivalent:
\begin{itemize}

\item[(i)]
$A$ is geometrically irreducible;

\item[(ii)]
$Q$ has at most one arrow;

\item[(iii)]
$A \cong K[X]/(X^m)$
for some $m \ge 1$.

\end{itemize}
\end{Cor}

\begin{Lem}\label{lem:reduction3}
Assume that $A$ is geometrically irreducible.
Then any cycle in $Q$ is a power of a loop.
\end{Lem}

\begin{proof}
Let $c := (a_1,\ldots,a_m)$ be a cycle.
If $c$ is primitive, then
$c$ is a loop by Lemma~\ref{lem:reduction1}.
Next, assume that $c$ is non-primitive.
Thus by Lemma~\ref{lem:primitivecycles} there is some
$1 \le q \le p \le m$ with $p-q < m-1$ such that
$(a_q,\ldots,a_p)$ is a primitive cycle and
$(a_1,\ldots,a_{q-1},a_{p+1},\ldots,a_m)$ is a cycle.
By induction $(a_1,\ldots,a_{q-1},a_{p+1},\ldots,a_m)$
is a power of a loop and by Lemma~\ref{lem:reduction1} $(a_q,\ldots,a_p)$
is also a loop.
Thus $c$ is a composition of loops, which necessarily all start (and end)
at the same vertex.
Now Lemma~\ref{lem:reduction2} implies that $c$ is a power of a
loop.
\end{proof}


\section{Locally free modules}


Assume that $A = KQ/\cI$ is geometrically irreducible.
By Lemma~\ref{lem:reduction3},
$A_i := e_iAe_i$ is a truncated polynomial ring for each vertex $i \in Q_0$.
An $A$-module $M$ is \emph{locally free} if
$e_iM$ is a free $A_i$-module for each $i$.
In this case, let $\rk(M) := (\rk(e_iM))_{i \in Q_0}$ be the \emph{rank vector} of $M$, where $\rk(e_iM)$ denotes the rank of the free $A_i$-module $e_iM$.

For an $A$-module $M$ let
$$
\supp(M) := \{ 1 \le i \le n \mid e_iM \not= 0 \}
$$
be the \emph{support} of $M$.

\begin{Lem}\label{lem:rigidlocfree1}
Assume that $A = KQ/\cI$ is geometrically irreducible.
Then every rigid $A$-module is locally free.
\end{Lem}

\begin{proof}
Let $R$ be a rigid $A$-module.
We claim that $R$ is locally free.
Assume not, then there exists some $m \ge 1$ and a locally free
$A$-module $M$ such that $\bd := \dimv(M) = \dimv(R^m)$.
Now the orbit of $R^m$ is open, so its closure is an irreducible
component of $\rep(A,\bd)$.
By upper semicontinuity, $M$ is not contained in this orbit closure.
Thus $\rep(A,\bd)$ is not irreducible, a contradiction.
\end{proof}

A finite-dimensional algebra $A$ is an \emph{Iwanaga--Gorenstein algebra} if $\pdim(D(A_A)) < \infty$ and
$\idim({_A}A) < \infty$.
In this case, \cite[Lemma~6.9]{AR} and \cite[Theorem~5]{Iw} imply that
$m := \pdim(D(A_A)) = \idim({_A}A)$, and that for
$M \in \md(A)$ the following are equivalent:
\begin{itemize}

\item[(i)]
$\pdim(M) < \infty$;

\item[(ii)]
$\idim(M) < \infty$;

\item[(iii)]
$\pdim(M) \le m$;

\item[(iv)]
$\idim(M) \le m$.

\end{itemize}

\begin{Prop}\label{prop:Gorenstein}
Assume that $A = KQ/\cI$ is geometrically irreducible.
Then $A$ is an Iwanaga--Gorenstein algebra.
Furthermore, for $M \in \md(A)$ the following are
equivalent:
\begin{itemize}

\item[(i)]
$\pdim(M) < \infty$;

\item[(ii)]
$\idim(M) < \infty$;

\item[(iii)]
$M$ is locally free.

\end{itemize}
\end{Prop}

\begin{proof}
We follow the lines of the proof of
\cite[Proposition~3.5]{GLS}.
By Lemma~\ref{lem:rigidlocfree1},
all projective and also all injective $A$-modules are
locally free.
We know from Lemma~\ref{lem:reduction1} that all cycles in $Q$ are loops.
Let $Q^\circ$ be the quiver obtained from $Q$ by removing all
loops.
Assume without loss of generality
that $Q_0 = \{ 1,\ldots, n \}$ and that for each
arrow $a$ of $Q^\circ$ we have $s(a) > t(a)$.

For each $1 \le i \le n$, the algebra $e_iAe_i$ is isomorphic a
truncated polynomial ring $K[\vep_i]/(\vep_i^{c_i})$ with $c_i \ge 1$.
Now $E_i := e_iAe_i$ is a uniserial $A$-module which is
free of rank $1$ as an $e_iAe_i$-module.

Since $Q^\circ$ does not have oriented cycles, $M \in \md(A)$ is
locally free if and only if there is a chain
$$
0 = M_0 \subset M_1 \subset \cdots \subset M_t = M
$$
of submodules such that for each $1 \le i \le t$ we have
$M_i/M_{i-1} \cong E_j$ for some $j$.
Let
$$
[M:E_j] := |\{ 1 \le i \le t \mid M_i/M_{i-1} \cong E_j \}| = \rk(e_jM).
$$

Let $P_1,\ldots,P_n$ be the indecomposable projective $A$-modules,
and let $I_1,\ldots,I_n$ be the indecomposable injective
$A$-modules corresponding to the vertices of $Q$.
Clearly, we have $E_1 \cong P_1$ and $E_n \cong I_n$.

Let $j$ be maximal such that $e_jM \not= 0$.
There is a short exact sequence
$$
0 \to \Omega(M) \to \bigoplus_{i=1}^j P(i)^{[M:E_i]} \to M \to 0.
$$
The subcategory of locally free modules is clearly
closed under kernels of epimorphisms and cokernels of
monomorphisms.

It follows that $\Omega(M)$ is locally free and that
$[\Omega(M):E_i] = 0$ for all $i \ge j$.
By induction, $\pdim(\Omega(M)) < \infty$.
This implies $\pdim(M) < \infty$.

Dually, one shows that $\idim(M) < \infty$ for all
locally free $M \in \md(A)$.

Since ${_A}A$ and $D(A_A)$ are locally free, we get that $A$ is
an Iwanaga--Gorenstein algebra.
In particular, this implies that (i) and (ii) are equivalent for all $M \in \md(A)$.

Finally,  assume that $M \in \md(A)$ is not locally free.
Thus there exists some $1 \le i \le n$ such that $e_iM$ is
not a free $e_iAe_i$-module.
Let
$$
0 \to \Omega(M) \to P \to M \to 0
$$
be a short exact sequence with $P$ projective.
It follows that
$$
0 \to e_i\Omega(M) \to e_iP \to e_iM \to 0
$$
is a short exact sequence of $e_iAe_i$-modules.
Now $e_iP$ is free as an $e_iAe_i$-module, but $e_iM$
is not.
This implies that also $e_i\Omega(M)$ is not free.
By induction this implies that $\pdim(M) = \infty$.
This finishes the proof.
\end{proof}


\section{Linear quivers}\label{sec:linear1}


\subsection{Linear quivers}\label{subsec:linear1}
Throughout Section~\ref{sec:linear1}, let
$Q = Q(n,t_1,\ldots,t_n)$ be the quiver
\[
\xymatrix{
0 \ar@(ul,ur)^{\vep_0}
&
1 \ar@(ul,ur)^{\vep_1} \ar@<-1ex>[l]_{\alpha_{11}} \ar@<1ex>^{\alpha_{t_11}}_{\cdots}[l]
&
2 \ar@(ul,ur)^{\vep_2} \ar@<-1ex>[l]_{\alpha_{12}} \ar@<1ex>^{\alpha_{t_22}}_{\cdots}[l]
&
\cdots \ar@<-1ex>[l]_{\alpha_{13}} \ar@<1ex>^{\alpha_{t3}}_{\cdots}[l]
&
n \ar@(ul,ur)^{\vep_n} \ar@<-1ex>[l]_{\alpha_{1n}} \ar@<1ex>^{\alpha_{t_nn}}_{\cdots}[l]
},
\]
where $n \geq 1$ and $t_1, \ldots, t_n \geq 1$.

Let $A = KQ/\cI$ with $\cI$ an admissible ideal in $KQ$.
We assume that $A$ is geometrically irreducible.

\subsection{Definition of a grading}
Define $\deg(\vep_i) := 0$ and $\deg(\alpha_{ij}) := 1$.
This definition extends to all paths and relations in the obvious way.
This turns $A$ into an $\N$-graded algebra.
If $R$ is a minimal set of relations generating $\cI$, then $R$ decomposes
as
$$
R = R_0 \cup R_1 \cup \cdots \cup R_n,
$$
where for $i \ge 1$, $R_i$ is the set of relations of degree $i$ in $R$.
In particular, we have
$R_0 = \{ \vep_0^{c_0},\ldots,\vep_n^{c_n} \}$ for some $c_i \ge 2$.

Define
$$
A_i := e_iAe_i = K[\vep_i]/(\vep_i^{c_i}).
$$

\subsection{Stratification by partitions}\label{subsec:strat}
Let $A = KQ/\cI$ be defined as before, and assume
that $\cI$ is generated by relations of degree one.
Let
$$
\pi\df \rep(A,\bd) \to \prod_{i=0}^n \rep(A_i,d_i)
$$
be the obvious projection.

\begin{Lem}\label{lem:strat}
For $(M_1,\ldots,M_n) \in \prod_{i=1}^n \rep(A_i,d_i)$ the preimage
$$
\pi^{-1}(M_1,\ldots,M_n)
$$
is an affine space whose dimension does not depend on the choice
of $(M_1,\ldots,M_n)$.
\end{Lem}

\begin{proof}
The representations $(M_1,\ldots,M_n,M_{\alpha_{ij}})$ of
$A$ have to satisfy relations of degree one.
When we fix $(M_1,\ldots,M_n)$ this amounts to solving systems
of linear equations given by the relations.
The result follows.
\end{proof}

As a consequence we get a stratification
$$
\rep(A,\bd) = \bigcup_{\cO} \cS(\cO)
$$
where the disjoint union is taken over all $G_\bd$-orbits $\cO$ in
$\prod_{i=0}^n \rep(A_i,d_i)$ and
$$
\cS(\cO) := \bigcup_{x \in \cO} \pi^{-1}(x).
$$
In the obvious way,
the orbits $\cO$ are parametrized by tuples
$\bp = (\bp_0,\ldots,\bp_n)$
of partitions $\bp_i$ of $d_i$ with entries at most $c_i$.
Let $\cO_\bp = \cO_{\bp_0,\ldots,\bp_n}$ denote the
orbit corresponding to $\bp$, and let $\cS_\bp = \cS_{\bp_0,\ldots,\bp_n} := \cS(\cO_\bp)$.
Let $\cO_\bd$ be the unique orbit of maximal dimension in $\prod_{i=0}^n \rep(A_i,d_i)$.
Then the closure $Z_\bd := \ov{\cS(\cO_\bd)}$ of $\cS(\cO_\bd)$ is an irreducible component of
$\rep(A,\bd)$.

\subsection{Relations of degree one and tensor algebras}
If $P$ is a projective $A$-module, then $P$ is locally free by
Lemma~\ref{lem:rigidlocfree1}.
It follows that
for each $1 \le i \le n$, the bimodule $B_i := e_{i-1}Ae_i$
is free both as a left $A_{i-1}$-module and as a right $A_i$-module.

We first investigate a structure of the algebra $A' := KQ/\cI'$,
where $\cI'$ is the ideal generated by the relations $\rho \in \cI$
of degree at most one.
Let
$$
S := A_0 \times \cdots \times A_n
\text{\;\;\; and \;\;\;}
B := B_1 \times \cdots \times B_n.
$$
Obviously, $B$ is an $S$-$S$-bimodule.
We denote by $T_S(B)$ the corresponding tensor algebra.

\begin{Prop} \label{prop isomorphism}
We have a natural isomorphism $A' \simeq T_S(B)$.
\end{Prop}

\begin{proof}
By the universal property of the path algebra, we have a natural epimorphism $K Q \to T_S(B)$, whose kernel contains $\cI'$, hence we get an epimorphism $A' \to T_S(B)$.
On the other hand, by the universal property of $T_S(B)$ we have an epimorphism $T_S (B) \to A'$.
Since both algebras are finite-dimensional, the claim follows.
\end{proof}

For each $i$ let $r_i$ be the rank of $B_i$ as a left $A_{i-1}$-module.
We have the following formulas for the rank vectors of projective
$A'$-modules.

\begin{Prop} \label{prop ranks}
If $i$ is a vertex of $Q$, then
\[
\rk(A'e_i) = (r_1 \cdots r_i, r_2 \cdots r_i, \ldots, r_i, 1, 0, \ldots, 0).
\]
\end{Prop}

\begin{proof}
By induction, it is sufficient to prove that $e_0 A' e_i$ has rank
$r_1 \cdots r_i$ over $A_0$.
Proposition~\ref{prop isomorphism} and induction implies that $e_0A'e_i \simeq B_1 \otimes_{A_1} e_1A'e_i$.
By induction $e_1A'e_i$ has rank $r_2 \cdots r_i$ over $A_1$.
By definition, the rank of $B_1$ over $A_0$ equals $r_1$.
Now the claim follows.
\end{proof}

\subsection{No relations of degree bigger than one}
We use the above to prove the following.

\begin{Prop}\label{prop:nobigrelations}
If $A$ is geometrically irreducible, then $R_2 = \cdots = R_n = \varnothing$.
\end{Prop}

\begin{proof}
By induction we may assume that $R_2 = \cdots = R_{n - 1} = \varnothing$, and to get a contradiction we assume that
$R_n \neq \varnothing$.

For a certain dimension vector $\bd$ we
construct two irreducible subsets $\cX$ and $\cY$ of $\rep (A,\bd)$ such that $\cX$ is an irreducible component of $\rep(A,\bd)$ and
$\dim \cY > \dim \cX$.

For $k,l \ge 0$ let
$$
P := Ae_n \oplus Ae_0^{k+l}
\text{\;\;\; and \;\;\;}
Q :=  (A e_{n - 1})^{r_n} \oplus (Ae_0)^k \oplus D(e_nA).
$$

Set $\bd = (d_0,\ldots,d_n) := \dimv(P)$, and let
$\cX$ be the closure of the orbit of $P$.
Since $P$ is projective, the orbit of $P$ is open.
Thus $\cX$ is an irreducible component of $\rep(A,\bd)$.

We have $\dim \cX = \dim G_{\bd} - \dim \End_A (P)$.
Furthermore, we easily see that $\dim \End_A (P) = d_n + (k + l)d_0$, hence
\[
\dim \cX = \dim G_{\bd} - d_n - kd_0 - ld_0.
\]
We have
$$
\rk(Ae_n) = (r_1 \cdots r_n-s,r_2 \cdots r_n,\ldots,r_n,1)
\text{\;\;\; and \;\;\;}
\rk(Ae_0) = (1,0,\ldots,0)
$$
where $s \ge 1$.
Here we used Proposition~\ref{prop ranks} and our assumptions on the sets $R_i$.

Thus we get
$$
\rk(P) = (r_1 \cdots r_n-s,r_2 \cdots r_n,\ldots,r_n,1) +
(k+l,0,\ldots,0).
$$

For $Q$ we get
$$
\rk(Q) = (r_1 \ldots r_n,r_2 \ldots,r_n,\ldots,r_{n-1}r_n,r_n,0)
+ (k,0,\ldots,0) + (0,\ldots,0,1).
$$
Here we used again Proposition~\ref{prop ranks}.

We get $\rk(P) = \rk(Q)$ if and only if
$-s+(k+l) = k$ if and only if $l = s$.
Thus, from now on let $l := s$.

Let $\cY$ be the closure of the orbit of $Q$.
Note that $\Hom (D(e_nA),(Ae_{n - 1})^{r_n} \oplus (A e_0)^k) = 0$,
since the modules $D(e_nA)$ and $(Ae_{n - 1})^{r_n} \oplus (A e_0)^k$
have disjoint supports.
Consequently, we get
$\dim \End_A(Q) = r_nd_{n-1} + kd_0 + d_n$, hence
\[
\dim \cY = \dim G_{\bd} - d_n - kd_0 - r_nd_{n-1}.
\]

Note that neither $r_n$, nor $l$ nor $d_{n - 1}$ depend on $k$.
Moreover, when $k$ increases, $d_0$ increases.
So we may choose $k$ such that $ld_0 > r_nd_{n - 1}$.
This finishes the proof.
\end{proof}

Let $Q^\circ = Q^\circ(n,t_1,\ldots,t_n)$
be the quiver obtained from $Q$ by deleting all loops.
Let
$$
A^\circ := KQ^\circ/\cI^\circ
$$
where $\cI^\circ$ is generated by all relations $\rho \in \cI$ such that
none of the loops occurs in $\rho$.

\begin{Cor}\label{cor:nobigrelations}
If $A^\circ$ is geometrically irreducible, then $R_2 = \cdots = R_n = \varnothing$.
\end{Cor}

\begin{proof}
Suppose that $A^\circ$ is geometrically irreducible.
Let $B$ be obtained from $A^\circ$ by adding a loop $\vep_i$ to each vertex $i$, and
by adding relations $\vep_i^{c_i}$ for some $c_i \ge 2$.
Obviously, $B$ is still geometrically irreducible.
Now the statement follows from Proposition~\ref{prop:nobigrelations}.
\end{proof}


\section{Quivers without loops}\label{sec:noloops}


\begin{Thm}\label{thm:2a}
Let $A = KQ/\cI$.
Assume that $Q$ does not contain any loop.
Then the following are equivalent:
\begin{itemize}

\item[(i)]
$A$ is geometrically irreducible;

\item[(ii)]
$A$ is hereditary.

\end{itemize}
\end{Thm}

\begin{proof}
As a consequence of Theorem~\ref{thm:intro1} we know that
(ii) implies (i).

To prove the converse, assume that $A = KQ/\cI$ is geometrically
irreducible.
We assume that $Q$ does not have loops.
Now Lemma~\ref{lem:reduction3} implies that $Q$ does not have any
oriented cycles.
Assume that $\cI \not= 0$.

Let $m$ be the minimum over the lengths of all paths occuring in
some relation in $\cI$, and let
$p = (a_1,\ldots,a_m)$ be such a path with $\length(p) = m$.
It follows that $m \ge 2$.

Let $Q'$ be the full subquiver of $Q$ with vertices the support
$\supp(p)$ of $p$.
For each relation
$$
\rho= \sum_{i=1}^t \lambda_ip_i
$$
in $\cI$ with $\lambda_i \not= 0$ for all $i$,
let
$$
\rho' := \sum_{i \in I_\rho} \lambda_ip_i,
$$
where $I_\rho := \{ 1 \le i \le t \mid p_i \text{ belong to } Q' \}$.
Now let $\cI'$ be generated by all $\rho'$, where $\rho$ runs through all
relations in $\cI$.

The varieties $\rep(A,\bd)$, where $\bd = (d_1,\ldots,d_n)$ is a
dimension vector with $d_i = 0$ for all $i \notin \supp(p)$,
are now isomorphic to the varieties $\rep(A',\bd')$ of
$A' = KQ'/\cI'$.
(Here $\bd'$ is obtained from $\bd$ be deleting the entries $d_i$ with
$i \in \supp(p)$.)

By minimality, the arrows $a$ occuring in the relations in $\cI'$ satisfy
$s(a) = s(a_i)$ and $t(a) = t(a_i)$ for some $i$.
Taking these arrows we obtain a geometrically irreducible
algebra $A'' = KQ''/\cI''$ with $Q'' \cong Q^\circ(m,t_1,\ldots,t_m)$
with $t_i \ge 1$ and $1 \le i \le m$, and by construction we
have $\cI'' \not= 0$.

Since there are no loops, all relations in $\cI''$ have degree at least two,
a contradiction to Corollary~\ref{cor:nobigrelations}.
This shows that $\cI = 0$, and therefore $A$ is hereditary.
\end{proof}

\begin{Cor}\label{cor:2b}
Let $A = KQ/\cI$.
Assume that $\gldim(A) < \infty$.
Then the following are equivalent:
\begin{itemize}

\item[(i)]
$A$ is geometrically irreducible;

\item[(ii)]
$A$ is hereditary.

\end{itemize}
\end{Cor}

\begin{proof}
Since by assumption $\gldim(A) < \infty$,
the No-Loop Theorem (cf. \cite{I} and \cite{L}) implies that the quiver $Q$
does not have any loops.
Now the statement follows from Theorem~\ref{thm:2a}.
\end{proof}

\bigskip
{\parindent0cm \bf Acknowledgements.}\,
The first author acknowledges the support of the National Science Center grant no.\ 2015/17/B/ST1/01731.
He would also like to thank the University of Bonn for its hospitality during his two one week visits in July 2016 and July 2017. 
The second author thanks the Nicolaus Copernicus University in Toru\'n
for one week of hospitality in October 2015, where this work was initiated.
He thanks for a second week in Toru\'n in March 2017, and he is also
grateful to the SFB/Transregio TR 45 for financial support.


\end{document}